\documentclass[preprint,12pt]{article}
\usepackage{amssymb}
\usepackage{amsfonts}
\usepackage{color}
\usepackage{graphics}
\usepackage{epsfig,psfrag,graphicx}
\usepackage{amssymb}
\usepackage{eepic,epic}

\usepackage{epsfig} 
\usepackage{amsmath} 
\usepackage{amssymb} 
\usepackage{amsbsy} 

\newtheorem{definition}{Definition}[section]
\newtheorem{theorem}{Theorem}[section]
\newtheorem{lemma}[theorem]{Lemma}

\newtheorem{proposition}[theorem]{Proposition}

\newcommand{\nc}{\newcommand}
\newcommand{\mA}{\mathcal{A}}
\nc{\V}{{\cal V}} \nc{\M}{{\cal M}} \nc{\T}{{\cal T}}
\def\u{u}
\nc{\D}{{\cal D}} \nc{\W}{\mathbb W} \nc{\ti}{\tilde}
\nc{\wti}{\widetilde} \nc{\vep}{\varepsilon}
\nc{\R}{{\mathbb R}} \nc{\N}{{\mathbb N}}

\nc{\di}{\displaystyle}

\nc{\pa}{\partial} \nc{\lra}{\longrightarrow}
 \nc{\weak}{\rightharpoonup}
\nc{\weakstar}{\stackrel{\ast}{\rightharpoonup}} \nc{\proof}{{\bf
Proof: }} \nc{\en}{^{\varepsilon,n}} \nc{\Rd}{{{\mathbb R}^{d}}}
\nc{\biting}{\stackrel{\,\,b}{\rightarrow}}

\renewcommand{\div}{{\mathrm{div}}\,}
\nc{\bB}{B}
\nc{\bS}{A}
\newcommand{\vrho}{\varrho}
\nc{\modular}[1]{{\stackrel{ #1}{\longrightarrow\,}}}
\newcommand{\Ref}[1]{{\rm(\ref{#1})}}
\nc{\vd}{\bar{v}} \nc{\zd}{\bar{z}}
\def\bbbone{{\mathchoice {\rm 1\mskip-4mu l}
{\rm 1\mskip-4mu l} {\rm 1\mskip-4.5mu l} {\rm 1\mskip-5mu l}}}
\title{Anisotropic  parabolic problems with slowly or~rapidly growing terms  }
\author{Agnieszka \'Swierczewska-Gwiazda\\[4ex]
{\small Institute of Applied Mathematics and Mechanics}, \\{\small University of Warsaw},\\
{\small Banacha 2, 02-097 Warsaw, Poland}\\ {\small \texttt{aswiercz@mimuw.edu.pl}}}
\begin{document}
\date{}
\maketitle
\begin{abstract}
We consider an abstract parabolic problem in a framework of maximal monotone graphs, possibly multi-valued with growth conditions formulated with help of an $x-$dependent $N-$function.
The main novelty of the paper consists in the lack of any growth restrictions on the ${ N}$--function combined with its anisotropic character, namely we allow the dependence on all the directions of the gradient, not only on its absolute value. This leads us to use the notion of modular convergence and studying in detail the question of density of compactly supported smooth functions with respect to the modular convergence. 

\medskip


\end{abstract}

\section{Introduction}
 
Our interest is directed to the phenomenon of anisotropic behaviour in a parabolic problem. 
The proposed approach allows for capturing very general form of growth conditions of a nonlinear term.  
We concentrate on an abstract parabolic problem.
Let $\Omega\subset\R^d$ be an open, bounded set with a ${\cal C}^2$~boun\-da\-ry~$\partial \Omega$,  $(0,T)$ be the time interval with $T<\infty$, $Q:=(0,T)\times\Omega$ and 
${\mathcal A}$ be a maximal monotone graph satisfying the assumptions (A1)--(A5) formulated below. 
Given $f$  and $u_0$ we want to find 
 $u:Q\to\R$ and $ A:Q\to\R^d$ such that 
\begin{align}\label{P1a}
u_t-\div A = f\quad&{\rm in}\ Q,\\
\label{P1aa}
(\nabla u,A)\in{\mathcal A}(t,x)\quad&{\rm in}\ Q,\\
\label{P2a}
u(0,x)=u_0\quad&{\rm in}\ \Omega,\\
\label{P3a}
u(t,x)=0\quad&{\rm on}\ (0,T)\times\partial\Omega.
\end{align}

The main objective of the present paper is to obtain existence result for the widest possible class of maximal monotone graphs. Hence various non-standard possibilities are considered including anisotropic growth conditions, $x-$dependent growth conditions and also relations given by maximal monotone graph. The last ones provide the possibility of generalization of discontinuous relations, namely considering $A$ as a discontinuous function of $\nabla u$, where the jumps of $A$  are filled by intervals creating vertical parts of the graph $\mA$.
Most of these generalities shall arise in a function that will prescribe the growth/coercivity conditions. Contrarty to the usual case of Leray-Lions type operators, where the polynomial growth is assumed,
e.g. $|A(\xi)|\le c(1+|\xi|)^{p-1},\,A(\xi)\cdot\xi\ge C|\xi|^p$ for some nonnegative constants $c, C$ and $p>1$ we shall 
work with $N-$functions.  By an $N-$function we mean that  $M:\bar \Omega\times\R^d\to\R_+$, $M(x,a)$
  is measurable w.r.t. $x$ for all $a\in\R^d$ and  continuous w.r.t. $a$ for a.a. $x\in\bar \Omega$,  convex in $a$, has superlinear growth, $M(x,a)=0$ iff $a=0$ and  
$$\lim_{|a|\to\infty}\inf_{x\in \Omega}\frac{M(x,a)}{|a|}=\infty.$$
 Moreover the conjugate function $M^*$ is defined as 
$$M^*(x,b)=\sup_{a\in{\mathbb R}^d}(b\cdot a-M(x,a)).$$
%
The graph is expected to satisfy for almost all $(t,x)\in Q$ the following set of assumptions:
\begin{enumerate}
\item[{ (A1)}] $\mA$  comes through the origin.
\item[{ (A2)}] $\mA$  is a monotone graph, namely
$$
(A_1-A_2)\cdot (\xi_1-\xi_2) \ge 0 \quad \textrm{ for all } (\xi_1, A_1),(\xi_2,A_2)\in \mA(t,x)\,.
$$
\item[{ (A3)}] $\mA$  is a maximal monotone graph. Let $(\xi_2, A_2)\in \R^d \times \R^d$.
\begin{equation*}\begin{split}
&\textrm{If } ( {A_1} - A_2)\cdot( {\xi_1} - \xi_2) \geq 0 \quad \textrm{ for all }
({\xi_1}, A_1) \in \mA(t,x)\\&
 \textrm{ then } (\xi_2, A_2) \in \mA(t,x).
\end{split}\end{equation*}
\item[{ (A4)}]  $\mA$  is an  {\it $M-$ graph.} There are non-negative $k\in
L^1(Q)$,   $c_*>0$ and $N$-function $M$ such that
\begin{equation*} 
A \cdot \xi \geq -k(t,x) +c_*(M(x,\xi) + M^*(x,A)) 
\end{equation*}
 for all $ (\xi,A)\in\mA(t,x).$ 
\item[{ (A5)}]  The existence of a measurable selection.  Either there is $\tilde A:Q\times \mathbb{R}^{d} \to
\mathbb{R}^{d}$ such that $(\xi,
\tilde A(t,x,\xi)) \in \mA(t,x)$ for all $\xi \in \Rd$ and
$\tilde A$ is measurable, 
 or there is $\tilde \xi:Q\times \mathbb{R}^{d} \to \mathbb{R}^{d}$ such
that $(\tilde\xi(t,x,A), A) \in \mA(t,x)$ for all $A \in \R^{d}$ and $\tilde\xi$ is measurable.
\end{enumerate}

Let us shortly refer again to the classical Leray-Lions operators. Within the setting  presented 
above  we would use the $N-$function
$M(a)=|a|^p$ with the conjugate function $M^*(a)=|a|^{p'}$, with $1/p+1/p'=1$. 

As we allow also for $x$ dependence, then the presented framework caputres also  the case of growth conditions in variable exponent case, namely
$M(a)=|a|^{p(x)}$. The further generalization is the anisotropic character and  functions different than only polynomials, hence the following example is acceptable 
$M(x,a)=a_1^{p_1(x)}\ln(|a|+1)+e^{a_2^{p_2(x)}}-1$ for $a=(a_1,a_2)\in\R^2$.
All the functions having a growth essentially different than polynomial (e.g. close to linear or exponential)
yield additional analytical difficulties and significantly constrain good properties of corresponding function spaces (like separability or reflexivity, or density of compactly supported smooth functions).  
We shall now discuss this issue in more detail. For this reason let us recall some definitions.
By the generalized Musielak-Orlicz class ${\mathcal L}_M(Q)$ we mean 
 the set of all measurable functions
 $\xi:Q\to\R^{d}$ for which the modular  $$\rho_{M,Q}(\xi)=\int_Q M(x,\xi(t,x)) \,dx\,dt$$ is finite. 
 By $L_M(Q)$ we mean the generalized Orlicz space which is the set of
all measurable functions
 $\xi:Q\to\R^{d}$ for  which $\rho_{M,Q}(\alpha\xi)\to0$ as $\alpha\to
 0.$
This is a Banach space with respect to the
 norm
$$\|\xi\|_M=\sup\left\{\int_Q \eta\cdot\xi dxdt : \eta\in L_{M^\ast}(Q),\int_Q M^\ast(x,\eta)
\,dx\,dt\le1\right\}.$$ 
 All over in the above definitions we used the notion of generalized Musielak-Orlicz spaces. Contrary to the classical Orlicz
 spaces we capture the case of $x-$dependent $N-$functions as well as functions dependent on the whole vector, not
 only on its absolute value (i.e. anisotropic).
Moreover, By $E_M(Q)$ we mean the closure of  bounded functions in
$L_M(Q)$. The space $L_{M^\ast}(Q)$ is the dual space of
$E_M(Q)$.
 A sequence $z^j$ is said to converge modularly to $z$ in $L_M(Q)$ if
there exists $\lambda>0$ such that
$$\rho_{M,Q}\left(\frac{z^j-z}{\lambda}\right)\to0$$
which is  denoted   by $z^j\modular{M} z$.
The basic estimates which we will frequently use in a sequel are 
the  H\"older inequality
\begin{equation}\label{hoelder}
\int_Q \xi \eta \,dx\,dt\le c\|\xi\|_M\|\eta\|_{M^*}
\end{equation}
and the  Fenchel-Young inequality
\begin{equation}\label{F-Y}
|\xi\cdot\eta|\le M(x,\xi)+M^*(x,\eta). 
\end{equation}
The essence of our considerations is  the lack of the assumption of $\Delta_2-$con\-di\-tion. We say 
that $M$ satisfies $\Delta_2-$condition if there exists a constant $c>0$ and a summable function $h$
such that
\begin{equation}\label{delta2}
M(x,2a)\le c M(x,a)+h(x)
\end{equation}
for all $a\in\R^d$. If $M$ satisfies \eqref{delta2} then $L_M(Q)$ is separable and compactly supported smooth functions are dense in strong topology. If additionally $M^*$ satisfies \eqref{delta2} then $L_M(Q)$ is reflexive. Notice that none of these assumptions is made in the present paper. For this reason the notion of modular topology and the issue of density of compactly supported smooth functions with respect to the modular topology are of crucial meaning. 
The basic properties which are mentioned above of anisotropic Musielak-Orlicz spaces were discussed and proved 
in~\cite{GwMiSw2013}.

As the density arguments become an essential tool, then the dependence of an $N-$function on $x$ 
becomes the significant constraint. The problem arises when we try to estimate uniformly the convolution operator. To handle this obstacle, we need some regularity with respect to the space variable. More precisely,  
we will assume that the function  $M$ satisfies  the following properties:
\begin{itemize}
\item[(M)] there exists a constant $H>0$ such that for all $x,y\in \Omega, |x-y|\le\frac{1}{2}$
and for all $\xi\in\R^d$ such that $|\xi|\ge 1$
\begin{equation}\label{log}
\frac{M(x,\xi)}{M(y,\xi)}\le |\xi|^\frac{H}{\log\frac{1}{|x-y|}}.
\end{equation}
Moreover, 
for every  bounded measurable set $G$  and every $z\in\R^d$
\begin{equation}\label{int}
\int_G M(x,z)< \infty.
\end{equation}
\end{itemize}

Below we formulate the definition and then state the existence theorem which is the main result of the present paper.  We shall use the following notation: by ${\cal C}_c^\infty(\Omega)$ we denote the space
of infinitely differentiable compactly supported functions in $\Omega$. Let $p\le 1\le \infty$ and 
$k\in\N$, then we denote by $(L^p(\Omega),\|\cdot\|_{L^p(\Omega)})$ the Lebesgue spaces and by 
$(W^{k,p}(\Omega), \|\cdot\|_{W^{k,p}(\Omega)})$ the Sobolev spaces.  By $W^{k,p}_0(\Omega)$  we
mean the closure of ${\cal C}_c^\infty(\Omega)$  with respect to the norm $ \|\cdot\|_{W^{k,p}(\Omega)}$ and $W^{-k,p'}(\Omega)$ with $1/p+1/p'=1$ denotes its dual space.  
Moreover we use the notation
${\cal C}_{weak}(0,T;L^2(\Omega))$ for the space of all $\varphi\in L^\infty(0,T;L^2(\Omega))$ which satisfy $(\varphi(t),v)\in{\cal C}([0,T])$ for all $v\in{\cal C}(\bar\Omega)$.

\begin{definition}\label{d:1}
Assume that
$
u_0\in L^{2}(\Omega)
$ and $f\in L^\infty(Q).$ 
We say that $(u,A)$ is weak solution to \eqref{P1a}-\eqref{P3a}  if
\begin{equation}\label{LM}
u\in L^\infty(0,T;L^2(\Omega)),  \nabla u \in L_M(Q),\ 
A\in  L_{M^\ast}(Q)
\end{equation}
and \begin{equation}\label{Cweak}
u \in  {\cal C}_{\textrm{weak}}(0,T; L^{2}(\Omega)).
\end{equation}
Moreover, the following identity 
\begin{equation}
\begin{split}
\int_Q \left(-u \varphi_t +A
\cdot \nabla \varphi \right) \,dx \,dt+\int_\Omega u_0(x)\varphi(0,x) \,dx
=\int_Qf\varphi \,dx\,dt,
\end{split} 
\end{equation}
is satisfied  for all  $\varphi\in{\cal C}_c^\infty((-\infty,T)\times\Omega)$
and
\begin{equation}
\left ( \nabla u ((t,x)), A(t,x) \right ) \in \mA(t,x)  \textrm{ for a.a. }
(t,x) \in Q.
\end{equation}
\end{definition}



\begin{theorem}\label{main2}
Let $M$ be an ${ N}$--function satisfying (M) and let  $A$ satisfy
conditions (A1)--(A5). Given $f\in L^\infty(Q) $
and
 $u_0\in
L^2(\Omega)$ 
there exists  a weak solution to \eqref{P1a}-\eqref{P3a}.
\end{theorem}
%
%
The current paper provides complementary studies to the results presented in~\cite{Sw2013}. Here we also consider the problem of existence of weak solutions to the parabolic problem including multivalued terms. However, the essential differerence consists in the properties of an 
$N-$function describing the growth conditions of graph $\mA$. In~\cite{Sw2013} we concentrated on 
the case with time-dependent $N-$function. This required more delicate approximation theorem 
and excluded the possibility of anisotropic functions. 
The studies presented here do not extend the results of the previous paper, but are parallel to them. We decided to omit here the dependence on time of an $N-$function, but added the possibility of  anisotropic behaviour.

The anisotropic parabolic problems were conisdered also  in~\cite{GwSw2010}. 
This was however much simpler situation, namely the studies concerned an equation and 
the $N-$function was assumed to be homogeneous in space. The anisotropic and space-inhomogeneous problems, however in 
slightly different setting, namely in the case of systems describing flow of non-Newtonian fluids were considered in~\cite{GwSw2008, GwSw2008a, GwSwWr2010, Wr2010}. The authors assumed $\Delta_2-$condition on the conjugate $N-$function. The simplified problem, namely the generalized Stokes equation, in the case omitting the $\Delta_2-$condition on the conjugate $N-$function was considered in~\cite{GwSwWr2012}.

The approach of maximal monotone graphs also to problems arising in fluid mechanics was undertaken in  \cite{ BuGwMaSw2009, 
GwMaSw2007} for the $L^p$ setting and in  \cite{BuGwMaRaSw2012,BuGwMaSw2012} for the setting in Orlicz spaces. The latter ones however were restricted to classical Orlicz spaces with the assumption that $\Delta_2-$condition was satisfied.


Most of the earlier results on existence of solutions to parabolic problems in non-standard setting concern
the case of classical Orlicz spaces, see e.g.~\cite{Donaldson}
and  later studies  of Benkirane, Elmahi and Meskine, cf.~\cite{BeEl1999, ElmahiMeskine, ElMe2005}.
 All of them concern the case of  an $N-$function dependent only on $|\xi|$ without the dependence on $x$.  

The paper is organized as follows: Section 2 contains the proof of Theorem~\ref{main2}, Section 3 
is devoted to the problems of density of compactly supported smooth functions with respect to 
the modular convergence. In the appendix we include some facts, which are used in the sequel and we refer 
to their proofs.

\section{Existence of solutions }

The current section contains a proof of Theorem~\ref{main2}. The construction of an approximate problem follows in two steps.
By (A5) there exists a measurable selection $\tilde A:Q\times \R^{d}\to\R^{d}$ of the graph $\mA$.  Obviously, each such a selection $\tilde A$ defined
on $\R^{d}$, is monotone and due to (A4) satisfies
\begin{equation}\label{316}
 \tilde A(t,x,\xi)\cdot\xi\ge -k(x,t)+ c_*(M(x,\xi)+M^*(x,\tilde A(t,x,\xi))\mbox{ for all }\xi \in \R^{d}.
\end{equation} We mollify $\tilde A$  with a smoothing kernel
and then construct the finite-dimensional problem by means of Galerkin method.
 Indeed, let 
 \begin{equation}\label{S}
 S\in
{\cal C}^\infty_c(\R^{d}), \  \int_{\R^{d}} S(y)\;dy=1, \, S(y)=S(-y), \   S_\varepsilon(y):=1/\varepsilon^{d}S(y/\varepsilon)
\end{equation}
  with ${\rm supp }\, S$ in a unit ball $B(0,1)\subset\R^d$ and define
 \begin{equation}\label{Te}
 A^\varepsilon (t,x,\xi):=(\tilde A* S_\varepsilon )(t,x,\xi)=\int_{\R^d}\tilde A(t,x,\zeta)
 S_\varepsilon (\xi-\zeta)\, d\zeta.
\end{equation}
 Using the convexity of $M$ and $M^*$ and the Jensen inequality allows to conclude that the approximation $ A^\varepsilon$ satisfies a condition analogous to  \eqref{316}, namely
\begin{equation}\label{319}
A^\vep \cdot \nabla u \geq -k(t,x) +c_*(M(x,\nabla u) + M^*(x,A^\vep)). 
\end{equation}
For the proof of analogous estimate for the approximation in case of polynomial conditions see~\cite{GwMaSw2007} and also~\cite{GwZa2007}. 

The assumption (A5) included either the possibility of  existence of a selection  $\tilde A$,
as was presented above, or  existence of a selection $\tilde\xi:Q\times\R^d\to\R^d$, such that 
$(\tilde \xi(t,x,A),A)$ is for all $A\in\R^d$ in the graph $\mA$. In the second case we would define
\begin{equation}
\xi^\vep(A):=(\tilde\xi*S_\vep)(t,x,A)+\vep A.
\end{equation}
Such a definition provides that the function $A\mapsto\xi^\vep(A)$ is invertible. Note that  since $\vep A\cdot A\ge 0$
one can show that for the pair $(\xi^\vep(A), A)$ an analogue of \eqref{319} holds, and consequently also for 
$(\xi, (\xi^\vep)^{-1}(\xi))$.
 Thus we may  define $A^\vep$ as follows
\begin{equation}
A^\vep:=(\tilde\xi*S_\vep+\vep\,  I\!d)^{-1}. 
\end{equation}
One  proceeds further analogously to the previous situation. In the sequel we present the proof for the case when there exists a selection
$\tilde\xi$ and $A^\vep$ is given by \eqref{Te}.

Consider now the basis consisting of eigenvectors of the Laplace operator  with Dirichlet boundary condtion and let  $u^{ \varepsilon,n}$ be the solution to the   finite dimensional problem with function~$A^\varepsilon$, namely 
$u^{ \varepsilon,n}(t,x):=\sum_{i=1}^nc_i^{ \varepsilon,n}(t)\omega_i(x)$ which solves the following system
\begin{equation}\label{aproksymacja}
\begin{split}
(u^{ \varepsilon,n}_t, \omega_i)+(A^\varepsilon(t,x,\nabla u^{ \varepsilon,n}),\nabla \omega_i)
=\langle f,\omega_i\rangle, \qquad i=1,\ldots,n, \\
u^{ \varepsilon,n}(0)=P^nu_0
\end{split}\end{equation}
where $P^n$ is the orthogonal projection of $L^2(\Omega)$ on the ${\rm span}\,\{\omega_1,
\ldots,\omega_n\}.$
 Let $Q^s:=(0,s)\times \Omega$ with 
$0<s<T$.
Using \eqref{319} allows to conclude
\begin{equation}\begin{split}
\sup\limits_{s\in(0,T)}\|u^{ \varepsilon,n}(s)\|_{L^2(\Omega)}^2&+c_*\int_QM(x,\nabla u^{ \varepsilon,n})
+M^*(x,A^\varepsilon(t,x,\nabla u^{ \varepsilon,n})) \,dx\,dt
\\&
\le c(\|u_0\|_{L^2(\Omega)}^2+\|f\|_{{L^\infty(Q)}}+\int_Q k \,dx\,dt
).\end{split}\label{energy-eps}
\end{equation}
 In a consequence of \eqref{energy-eps}  there exists a subsequence (labelled the same) such that
   \begin{equation}\label{ep-to-zero}\begin{split}
  \nabla u^{ \varepsilon,n}\weakstar  \nabla u^{ n}\quad &\quad\mbox{weakly-star  in}\quad L_{M}(Q),\\
A^\varepsilon (\cdot,\cdot,\nabla u^{ \varepsilon,n})\weakstar A^{ n} \quad &\quad\mbox{ weakly-star  in}\quad L_{M^*}(Q).
\end{split}\end{equation}
Moreover, from relation \eqref{aproksymacja} we concude the boundedness of the sequence 
$u^{ \varepsilon,n}_t$ in $L_{M^*}(Q)$ and hence up to the subsequence we have 
\begin{equation}
u^{ \varepsilon,n}_t\weakstar u^{ n}_t\quad\quad\mbox{ weakly-star  in}\quad L_{M^*}(Q).
\end{equation}

Further we observe that  \eqref{aproksymacja}  implies  that 
$\frac{d}{dt}c_i^{\vep,n}(t)$ is bounded in the space $L_{M^*}([0,T])$, what implies the  uniform integrability in $L^1([0,T])$. Consequently  there exists a monotone, continuous $L:\R_+\to\R_+$, with $L(0)=0$ such that for all $s_1,s_2\in (0,T)$
$$\left|\int_{s_1}^{s_2}\frac{d}{dt}c_i^{\vep,n}(t)\,dt\right|\le L(|s_1-s_2|)$$
and thus the sequence $c_i^{\vep,n}$ is uniformly equicontinuous 
$$|c_i^{\vep,n}(s_1)-c_i^{\vep,n}(s_2)|\le L(|s_1-s_2|).$$
From   \eqref{energy-eps} we conclude that  $c_i^{\vep,n}(t) $ is bounded in $L^\infty([0,T])$ 
and hence by the Arzel\`a-Ascoli theorem there  exists a uniformly convergent subsequence 
$\{c_i^{\vep_k,n}\}$ in ${\cal C}([0,T])$  and taking into account the regularity of the basis 
$\{\omega_i\}_{i=1}^n$ we conclude

\begin{equation}\label{mocna}
 u^{ \varepsilon,n}\to  u^{ n}\quad \quad\mbox{strongly in}\quad {\cal C}([0,T];{\cal C}^1(\overline\Omega)).
\end{equation}
The limit passage with $\varepsilon\to 0$ is done on the level of finite-dimensional problem. It follows the similar lines as in~\cite{BuGwMaSw2012}, however we shall recall the main steps. 
Using \eqref{ep-to-zero}-\eqref{mocna} we obtain the  following limit problem 
\begin{equation}\label{Galerkin}
\begin{split}
(u^{ n}_t, \omega_i)+(A^{ n},\nabla \omega_i)
=\langle f,\omega_i\rangle, \qquad i=1,\ldots,n, \\
u^{ n}(0)=P^nu_0.
\end{split}\end{equation}
To complete the limit passage we need to provide that  
\begin{equation}\label{in-graph}
( \nabla u^{ n}, A^{ n})\in{\cal A}.
\end{equation}
Following \cite{BuGwMaSw2012} and also \cite{Sw2013}, with simple algebraic tricks and estimates which are not included in the present paper,
we conclude that for all $B\in\Rd$ and for a.a. $(t,x)\in Q$
\begin{equation}
(A^{ n}-\tilde A(t,x,B)) \cdot (\nabla u^{ n}-B)\ge0 \,.
\end{equation}
Hence, using  the equivalence
  of $(i)$ and $(ii)$ in  Lemma~\ref{LS*}, we arrive to \eqref{in-graph}.
Before passing to the limit with $n\to\infty$ we notice that 
in the same manner as before we obtain the estimates, which are uniform with respect to $n$, namely 
\begin{equation}\begin{split}
\sup\limits_{s\in(0,T)}\|u^{ n}(s)\|_{L^2(\Omega)}^2+\int_QM(x,\nabla u^{ n})+M^*(x,A^{ n}) \,dx\,dt\\
\le c(\|u_0\|_{L^2(\Omega)}^2+\|f\|_{L^\infty(Q)}+\|k\|_{L^1(Q)}). 
\end{split}\label{esty2}\end{equation}
Consequently  there exists a subsequence, labelled the same, such that 
  \begin{equation}\begin{split}
 \nabla u^{ n}\weakstar \nabla u\quad &\quad\mbox{weakly-star in}\quad L_M(Q),\\
  u^{ n}\weak  u\quad &\quad\mbox{weakly in}\quad L^1(0,T;W^{1,1}(\Omega)),\\
A^{ n}\weakstar A\quad &\quad\mbox{ weakly-star  in}\quad L_{M^*}(Q),\\
u^{ n}\weakstar u\quad&\quad\,\mbox{weakly-star in}\,\,L^\infty(0,T;L^2(\Omega)).\\
u^{ n}_t\weakstar u_t\quad&\quad\mbox{ weakly-star  in}\quad W^{-1,\infty}(0,T;L^2(\Omega)).\end{split}\label{228}\end{equation}
Using \eqref{228} we let $n\to\infty$ and conclude from \eqref{Galerkin} that the following  identity
\begin{equation}\label{limit}
 u_t  -\div {A} =  f
\end{equation}
holds in a distributional sense. 
Again, to complete the limiting procedure, we need to show that $(\nabla u, A)\in\mA(t,x)$. 
This case however requires more attention, contrary to the previous limit passage on the level of
fixed finite dimension $n$. The essence of this step is using the maximal monotonicity of the graph
$\mA$, in particular the property formulated in Lemma~\ref{Minty2}. As the  assumptions 
 \eqref{11}-\eqref{33} are obviously satisfied, then our attention shall be directed to \eqref{44}. 
For this aim we need to establish a strong energy inequality. Since testing \eqref{limit} with a solution is not possible, we first approximate it with respect to the space variable. 
By  Theorem~\ref{main1} there exists a sequence $v^j\in L^\infty(0,T;{\cal C}_c^\infty(\Omega
))$ such that 
\begin{equation}\label{zb_w_M}
\nabla  v^j\modular{M} \nabla u\ \mbox{ modularly in}\  L_M(Q) 
\ \mbox{and}\  v^j\to u \ \mbox{strongly in }\ 
L^2(Q).
\end{equation}
And hence we shall test with a  function of the form 
\begin{equation}\label{testfunkcja}
u^{j,\epsilon}= K^\epsilon\ast (K^\epsilon\ast v^j \bbbone_{(s_0,s)})
\end{equation}
   with $ K\in
{\cal C}_c^\infty(\R),$  $ K(\tau)= K(-\tau),\ 
\int_\R K(\tau)d\tau=1$ and  defining
$ K^\epsilon(t)=\frac{1}{\epsilon} K(t/\epsilon)$, $\epsilon<\min\{ s_0,T-s\}. $ Thus
\begin{equation}\label{335}
\int_{s_0}^{s}\int_\Omega (u*K^\epsilon)\cdot \partial_t(v^j*K^\epsilon)\,dx\,dt
=\int_Q A\cdot\nabla u^{j,\epsilon}\,dx\,dt-\int_Q fu^{j,\epsilon}\,dx\,dt.
\end{equation}
Because of \eqref{zb_w_M} we easily pass to the limit with $j\to\infty$.
Indeed, the left-hand side of \eqref{335} can be easily handled since this term can be reformulated
to $\int_Q((\partial_t K^\epsilon)*K^\epsilon*u) v^j\,dx\,dt$ and hence the limit passage is obvious.
Note that for all $0<s_0<s<T$ it follows
	\begin{equation}
	\begin{split}
 	 \int_{s_0}^{s}\int_\Omega
	(K^\epsilon\ast u)\cdot\partial_t(K^\epsilon\ast u) \,dx\,dt
	 &
= 
	\int_{s_0}^{s} \frac{1}{2}\frac{d}{d t}\| K^\epsilon\ast
	u\|^2_{L^2(\Omega)}\,d t\\
	&=\frac{1}{2} \| K^\epsilon\ast u(s)\|^2_{L^2(\Omega)}-
	\frac{1}{2}\| K^\epsilon\ast u(s_0)\|^2_{L^2(\Omega)}.
	\end{split}
	\end{equation}
Passing  to the limit with $\epsilon\to0$ yields  for almost
all $s_0, s$, namely for all Lebesgue points of the function $ u(t)$ that the following identity
	\begin{equation}\label{vt}
	\lim\limits_{\epsilon\to0}
	\int_{s_0}^{s}\int_\Omega( u\ast  K^\epsilon)\cdot \partial_t ( u\ast K^\epsilon)=
	\frac{1}{2}\| u(s)\|_{L^2(\Omega)}^2-\frac{1}{2}\| u(s_0)\|^2_{L^2(\Omega)}
	\end{equation}
holds.
Observe now the term
	$$\int_0^T\int_\Omega  A \cdot( K^\epsilon\ast(( K^\epsilon\ast \nabla u)\ 	\bbbone_{(s_0,s)}))dxdt
	=\int_{s_0}^{s}\int_\Omega
	( K^\epsilon\ast A )\cdot( K^\epsilon\ast \nabla u) dxdt.$$
Both of the sequences $\{ K^\epsilon\ast A \}$ and
$\{ K^\epsilon\ast \nabla u\}$ converge in measure in $Q$ by Proposition~\ref{ae}.
 Moreover
	$$\int_Q (M(x,\nabla u)+ M^\ast(x, A ))\,dx\,dt< \infty	.
	$$ 
 Hence by Proposition~\ref{sup} we conclude 
that the sequences
$\{M^*\left( x,K^\epsilon\ast A \right)\}$ and 
$\{M\left( x,K^\epsilon\ast \nabla u\right)\}$
are uniformly integrable and with help of 
Lemma~\ref{modular-conv} we have
	\begin{equation*}
	\begin{split}
	 K^\epsilon\ast \nabla u
	&\modular{M}\nabla u\quad\mbox{modularly in }L_{M}(Q),\\
	 K^\epsilon\ast A 
	&\modular{M^\ast} A \quad\quad\mbox{modularly in }L_{M^*}(Q).
	\end{split}
	\end{equation*}
Applying Proposition~\ref{product} allows to conclude
	\begin{equation}\label{chi}
	\lim\limits_{\epsilon\to0}\int_{s_0}^{s}\int_\Omega
	( K^\epsilon\ast A )\cdot( K^\epsilon\ast \nabla u) dxdt
	=\int_{s_0}^{s}\int_\Omega  A \cdot \nabla u dxdt.
	\end{equation}
Passing to the limit with $\epsilon\to0_+$ in the right-hand side  is obvious. 
Hence for the moment we are able to claim that the following holds
\begin{equation}
\frac{1}{2}\|u(s)\|_2^2-\frac{1}{2}\|u(s_0)\|_2^2+\int_{Q^s}A\cdot \nabla u \,dx\,dt=\int_{Q_s} fu \,dx\,dt
\end{equation}
 for almost all $0<s_0<s<T$. For further considerations we need to know that the same holds for $s_0=0$, hence let us
 pass to the limit with $s_0\to0$. Thus, we need to establish that \eqref{Cweak} holds.  
  We shall observe that using the approximate equation we estimate 
the sequence $\{\frac{d\u^n}{dt}\}$ uniformly (with respect to $n$) in the space $L^1(0,T;W^{-r,2}(\Omega))$, where
 $r>\frac{d}{2}+1$.
Consider 
 $\varphi\in L^\infty(0,T;W^{r,2}_0(\Omega))$, $\|\varphi\|_{L^\infty(0,T;W^{r,2}_0)}\leq 1$ and observe that 

	\begin{equation*}
	\left\langle\frac{du^n}{dt}, \varphi \right \rangle=
	\left\langle\frac{du^n}{dt}, P^n\varphi\right \rangle
	= -\int_\Omega
	 A^n\cdot \nabla(P^n\varphi)\,dx\\
	 +\int_\Omega f\cdot P^n\varphi\,dx.
	\end{equation*}
Since the orthogonal projection is continuous in  ${W^{r,2}_0(\Omega)}$ and 
$W^{r-1,2}(\Omega)\subset L^\infty(\Omega)$ we estimate as follows 
	\begin{equation}\label{dt1}
	\begin{split}
	&\Big{|}\int_0^T\int_\Omega  A^n\cdot \nabla(P^n\varphi)
	dxdt\Big{|}\le \int_0^T\|
	 A^n\|_{L^1(\Omega)}\|\nabla(P^n\varphi)\|_{L^\infty(\Omega)}dt\\
	&\le c\int_0^T\| A^n\|_{L^1(\Omega)}\|P^n\varphi\|_{W^{r,2}_0}dt
	\le c
	\| A^n\|_{L^1(Q)}\|\varphi\|_{L^\infty(0,T;W^{r,2}_0)}.
\end{split}	\end{equation}
From \eqref{esty2} and Lemma~\ref{uni-int} we conclude there exists   
a monotone, continuous function  $L:\R_+\to\R_+$, with $L(0)=0$, 
 independent of $n$, such that 
	$$\int_{s_1}^{s_2}\| A^n\|_{L^1(\Omega)}
	\le L(|s_1-s_2|)$$
for all $s_1,s_2\in[0,T]$. Conseqently,  \Ref{dt1} gives us
	$$\left|\int_{s_1}^{s_2}\left\langle\frac{d u^n}{dt},\varphi\right\rangle dt\right|	\le
	L(|s_1-s_2|)$$
for all $\varphi$ with 
${\rm supp}\ \varphi\subset(s_1,s_2)\subset[0,T]$ and 
$\|\varphi\|_{L^\infty(0,T;W^{r,2}_0)}\le 1$. 
Since 
	\begin{equation}\begin{split}
	\|u^n(s_1)-u^n(s_2)\|_{W^{-r,2}}
	=
	\sup\limits_{\|\psi\|_{W^{r,2}_0}\le1}\left|\left\langle
	\int_{s_1}^{s_2}\frac{du^n(t)}{dt},\psi\right\rangle\right |
	\end{split}\end{equation}
then
	\begin{equation}
	\label{equi}
	\sup\limits_{n\in\N}\|u^n(s_1)-u^n(s_2)\|_{W^{-r,2}}\le L(|s_1-s_2|),
	\end{equation}
namely the family of functions $\u^n:[0,T]\to W^{-r,2}(\Omega)$ is equicontinuous. 
Moreover, it is  uniformly bounded in $L^\infty(0,T;L^2(\Omega))$ and hence the sequence $\{\u^n\}$ is relatively compact 
in  ${\cal C}([0,T];W^{-r,2}(\Omega))$ and the limit $\u\in {\cal C}([0,T];W^{-r,2}(\Omega))$.  Thus there exists  a sequence 
$\{s_0^i \}_i$, 
$s_0^i \to 0^+$ as $i\to\infty$ such that 
%
%
	\begin{equation}
	\u(s_0^i){{\stackrel{i\to \infty}{\longrightarrow\,}}}\u(0)\quad\mbox{in }W^{-r,2}(\Omega).
	\end{equation}
The limit above coincides with the weak limit of $\{\u(s^i_0)\}$ in $L^2(\Omega)$
what allows to claim that 
%
%
	\begin{equation}\label{liminfu0}
	\liminf\limits_{i\to\infty}\|\u(s_0)\|_{L^2(\Omega)}\geq\|\u_0\|_{L^2(\Omega)}.
	\end{equation} 
We obtain from \eqref{Galerkin} for any   Lebesgue point $s$  of $\u$ that 
	\begin{equation}\label{777}
	\begin{split}
	\limsup\limits_{n\to\infty}\int_{Q_s} &A^n\cdot \nabla\u^n \,dx\,dt\\
	& = 
	 \frac{1}{2} \|\u_0\|^2_{2} -
	\liminf\limits_{k\to\infty} \frac{1}{2}  \|\u^n(s)\|^2_{2}
	+\lim\limits_{n\to\infty}\int_{Q_s} f u^n\,dx\,dt\\
	& \leq
	 \frac{1}{2} \|\u_0\|^2_{2} -
	 \frac{1}{2}  \|\u(s)\|^2_{2}+\int_{Q_s} f u\,dx\,dt\\
	 &\leq
	 \liminf\limits_{i\to \infty}\left( \frac{1}{2} \|\u(s^i_0)\|^2_{2} -
	 \frac{1}{2}  \|\u(s)\|^2_{2}\right)+\int_{Q_s} f u\,dx\,dt\\
	  &{=} \lim\limits_{i\to\infty} 
	  \int_{s^i_0}^s\int_\Omega A\cdot\nabla\u \,dx\,dt \\
	  &= \int_{0}^s\int_\Omega A\cdot\nabla\u \,dx\,dt
	 \end{split}
	\end{equation}
	what provides that  \eqref{44} is satisfied and  Lemma~\ref{Minty2} allows to complete the proof. 
	
\section{Approximation}
In this section we shall concentrate on the issue of density of compactly supported smooth functions with respect to the modular topology. The fundamental studies in this direction are due to Gossez for the case of classical Orlicz spaces and elliptic equations \cite{Gossez1, Gossez2}.
The similar considerations for isotropic $x-$dependent $N-$functions are due to Benkirane et al. cf.~\cite{Be2011}, see also \cite{GwMiWr2012} for anisotropic case with an application to elliptic problems. 
 Note that the main idea is analogous to \cite{GwMiWr2012}. However, Gwiazda et al. approxi\-ma\-te the truncated functions which are appropriate test functions in the con\-si\-dered elliptic equation. This is not the case of parabolic problems. Hence the presented approximation theorem is under weaker assumptions and the dependence on time is taken into account. 
Since  this  result is  essential  for proving existence of weak solutions, then we include the details 
for completeness. 

\begin{theorem}\label{main1}
If $u\in  L^\infty(0,T;L^2(\Omega))\cap L^1(0,T;W^{1,1}_0(\Omega)), \nabla u\in L_M(Q)$ 
then there exists a sequence $v^j \in L^\infty(0,T;{\cal C}_c^\infty(\Omega))$
satisfying 
\begin{equation}\label{zb_w_M1}
\nabla v^j\modular{M} \nabla u\ \mbox{ modularly in}\  L_M(Q) 
\ \mbox{and}\  v^j\to u \ \mbox{strongly in }\ 
L^2(Q).
\end{equation}
\end{theorem}

\begin{proof}
Already for Lipschitz domain $\Omega$ there exists a finite family of star-shaped Lipschitz domains $\{\Omega_i\}$ such that 
$$\Omega=\bigcup\limits_{i\in J}\Omega_i,$$
cf.~\cite{Novotny}. We introduce the partition of unity $\theta_i$ with $0\le\theta_i\le1,\, \theta_i\in {\cal C}^\infty_c(\Omega_i), \,
{\rm supp} \,\theta_i=\Omega_i, \sum_{i\in J}\theta_i(x)=1$ for $x\in\Omega$
and  define the truncation operator $T_\ell(u)$  as follows 
\begin{equation}
T_\ell(u)=\left\{
\begin{array}{rcl}
u&{\rm if}&|u|\le \ell,\\
\ell&{\rm if}& u>\ell,\\
-\ell&{\rm if}& u< -\ell.
\end{array}\right.\end{equation}
Define $Q_i:=(0,T)\times \Omega_i$. Obviously $$T_\ell(u)\in L^\infty(0,T;L^2(\Omega))\cap L^1(0,T;W^{1,1}_0(\Omega)), \nabla T_\ell u\in L_M(Q)$$ and for each 
$i\in J$
$$\theta_i\cdot T_\ell(u)\in L^\infty(Q_i)\cap L^1(0,T;W^{1,1}_0(\Omega_i))\cap L^\infty(0,T;L^2(\Omega_i)) .$$
 Introducing the truncation of $u$ was necessary to provide that 
$$\nabla T_\ell(u)\cdot \theta_i + T_\ell(u)\cdot \nabla \theta_i= \nabla (T_\ell(u)\cdot \theta_i) \in
 L_M(Q_i).$$
Without loss of generality assume that all $\Omega_i$ are  star-shaped domains with respect to a ball of radius R, i.e. $B(0,R)$.
We define for $(t,x)\in(0,T)\times\Omega$
\begin{equation}\label{Sdelta2}
\begin{split}
 {\cal S}_\delta (\theta_iT_\ell(u))(t,x)
:=\frac{1}{\left(1-{\delta}/{R}\right)}\int_Q
  S_\delta(x-y)\theta_i T_\ell(u)\left(t,\left(1-{\delta}/{R}\right)y\right)  \,dy.
\end{split}\end{equation}
Our aim is to show that there exists a constant $\lambda>0$ such that
\begin{equation}\label{216}
\lim\limits_{l\to\infty}\lim\limits_{\delta\to 0_+}\varrho_{M,Q_i}\left( \frac{\nabla u -\nabla {\cal S}_\delta(\theta_i T_\ell(u))}{\lambda} \right)
=0.
\end{equation}
For this purpose we introduce a sequence of simple functions
$$\xi_n(t,x):=\sum_{j=1}^n \alpha_j^n \bbbone_{G_j}(t,x),\quad  \alpha_j^n \in\R,\ \bigcup_{j\in\{1,\ldots,n\}} G_j=Q$$ 
which converges to $ \nabla (\theta_i\cdot T_\ell(u)) $ modularly in $L_M(Q)$. 
Moreover, let $\lambda_0, \lambda_1, \lambda_2$ be some appropriate constants which we specify later such that the following estimate holds
\begin{equation}
\begin{split}
\varrho_{M,Q_i}&\left( \frac{\nabla u -\nabla {\cal S}_\delta(\theta_i T_\ell(u))}{\lambda} \right)\\
&\le \frac{\lambda_0}{\lambda}
\rho_{M,Q_i}\left(\frac{ {\cal S}_\delta  \nabla(\theta_iT_\ell( u))- {\cal S}_\delta \xi_n}{\lambda_0}\right)
+ \frac{\lambda_0}{\lambda}
\rho_{M,Q_i}\left(\frac{\nabla (\theta_iT_\ell( u))-\xi_n}{\lambda_0}\right)\\
&\quad+\frac{\lambda_1}{\lambda}\rho_{M,Q_i}\left(\frac{ {\cal S}_\delta \xi_n-\xi_n}{\lambda_1}\right)
+\frac{\lambda_2}{\lambda}\varrho_{M,Q_i}\left( \frac{ \nabla u -\nabla (T_\ell(u)\theta_i)}{\lambda_2}\right)\\
&=I_1+I_2+I_3+I_4.
\end{split}\end{equation}
Consider first $I_3$. 
The existence of a sequence $\xi_n$ is provided by Lemma \ref{lem-dense}.
Let 
$B_\delta:=\{y\in \Omega\ :\ |y|<\delta\}$.
Then 
\begin{equation}
\begin{split}
{\cal S}_\delta \xi_n-\xi_n
=\int_{B_\delta}  S_\delta(y)\sum\limits_{j=1}^n\left(\alpha_j^n\bbbone_{G_j}
(t, (1-\delta/R)(x-y))-
\alpha_j^n\bbbone_{G_j}(t, x)\right) \,dy
\end{split}\end{equation}
and the  Jensen inequality and Fubini theorem yield
\begin{equation}\label{proste}\begin{split}
&\rho_{M,Q_i}\left(\frac{ {\cal S}_\delta \xi_n(t,x)-\xi_n}{\lambda_1}\right)\\
&=\int_Q M(x,\frac{1}{\lambda_1}\int_{B_1}  S(y)\sum\limits_{j=1}^n(\alpha_j^n\bbbone_{G_j}
(t, (1-\delta/R)(x-\delta y))\\&\qquad-
\alpha_j^n\bbbone_{G_j}(t, x)) \,dy) \,dt\,dx\\
&\le \int _{B_1}  S(y)( \int_Q  M(x,\frac{1}{\lambda_1} \sum\limits_{j=1}^n\alpha_j^n(\bbbone_{G_j}
(t, (1-\delta/R)(x-\delta y))\\&\qquad-
\bbbone_{G_j}(t, x))) \,dt\,dx) \,dy.\\
\end{split}\end{equation}
Note that  $\{\frac{1}{\lambda_1} \sum\limits_{j=1}^n\alpha_j^n\left(\bbbone_{G_j}
(t, (1-\delta/R)(x-\delta y))-
\bbbone_{G_j}(t, x)\right) \,dt\,dx) \}_{\delta>0}$ converges a.e. in $Q$ to zero as $\delta\to0_+$ and 
\begin{equation}\begin{split}\label{estimate2}
M(x, \frac{1}{\lambda_1} \sum\limits_{j=1}^n\alpha_j^n\left(\bbbone_{G_j}
(t, (1-\delta/R)(x-\delta y))-
\bbbone_{G_j}(t, x)\right)\\
\le \sup\limits_{|z|=1}M(x,\frac{1}{\lambda_1} \sum\limits_{j=1}^n\alpha_j^n z).
\end{split}\end{equation}
Assumption \eqref{int} provides that  the right-hand side of \eqref{estimate2} is integrable, hence the Lebesgue dominated convergence theorem allows to conclude that $I_3$ vanishes as $\delta\to0_+$. 
Lemma~\ref{modular-topology}  allows to estimate $I_1$ on each $\Omega_i$ as follows
\begin{equation}
I_1=\frac{\lambda_0}{\lambda}\rho_{M,Q_i}\left(\frac{ {\cal S}_\delta  (\nabla(\theta_iT_\ell( u))- \xi_n)}{\lambda_0}\right)
\le c \rho_{M,Q_i}\left(\frac{ \nabla(\theta_iT_\ell( u))- \xi_n}{\lambda_0}\right)
\end{equation}
and hence by Lemma~\ref{lem-dense} there exists a constant $\lambda_0$ such that 
$$\lim\limits_{n\to\infty} (I_1+I_2)=0.$$
Moreover, as $\ell\to\infty$ we observe the following convergence
$$T_\ell(u)\to u\ \ \  {\rm strongly \,in}\ \ \   L^1(0,T;W^{1,1}_0(\Omega))$$
and hence also, at least for a subsequence, almost everywhere.
To find a uniform estimate we observe  that 
 $M(x,\nabla T_\ell(u(t,x)))\le M(x,\nabla u(t,x))$ a.e. in Q. Indeed,  $T_\ell(u)$ and $u$ coincide for $|u|\le \ell $ and on the remaining two sets, where $T_\ell (u)$
 is equal to $\ell$ or $-\ell$ we have that
 $ T_\ell(u)\in L^1(0,T;W^{1,1}_0(\Omega))$, then $\nabla T_\ell(u)$  is almost everywhere equal to zero.
  Consequently $M(x,\nabla T_\ell(u(t,x)))$ is uniformly 
integrable, which combined with pointwise convergence provides
$$\nabla T_\ell(u)\to \nabla u\ \ \  {\rm modularly\  in}\ \ \  L_M(Q)$$
as $\ell\to\infty$, hence there exists a constant $\lambda_2$ such that $\lim_{\ell\to\infty} I_4=0$.
Finally, choosing $\lambda>\max\{3 \lambda_0, 3\lambda_1, 3\lambda_2\}$, passing first with $\delta\to0_+$, then 
$n\to\infty$ and $\ell\to\infty$ we arrive to \eqref{216}.

The strong convergence in $L^2$ is straighforward, since an $N-$function $M(x,a)=|a|^2$ satisfies 
$\Delta_2-$condition and the strong and modular convergence coincide.

\end{proof}

\begin{lemma}\label{modular-topology} Let an $N-$function satisfy condition (M),
  $S$ and $S_\delta$ be given by \eqref{S}
and  assume that $\Omega$ is a star-shaped domain  with respect to a ball centered at the origin $B(0,R)$ for some $R>0$.
We define  the family of operators
\begin{equation}\label{Sdelta}
 {\cal S}_\delta z(t,x):=\left(1-{\delta}/{R}\right)^{-1}\int_\Omega
  S_\delta(x-y)z\left(t,\left(1-{\delta}/{R}\right)y\right)  \,dy.
\end{equation}
Then there exists a constant $c>0$ (independent of $\delta$) such that 
\begin{equation}\label{cont}
\int_Q M(x, {\cal S}_\delta z(t,x))\,dx\,dt\le c\int_Q M(x,z(t,x))\,dx\,dt
\end{equation}
holds for every $z\in L_M(Q)\cap L^\infty(0,T;L^1(\Omega))$.
\end{lemma}

\begin{proof}
Since $\Omega$ is
a star-shaped domain with respect to  $B(0,R)$, then for each $\lambda\in(0,1)$
$$(1-\lambda)x+\lambda y\in \Omega\quad\mbox{for each }x\in \Omega, y\in B(0,R).
$$
Hence 
for $\delta<R$ we may choose $\lambda=\delta/R$ and conclude that 
$$\left(1-\frac{\delta}{R}\right)\Omega+\delta B(0,1)\subset \Omega.$$

Let ${\cal S}_\delta z(t,x)$ be defined by~\eqref{Sdelta}. 
Since $\overline{\left(1-\frac{\delta}{R}\right)\Omega+\delta B(0,1)}\subset \Omega,$ then 
it holds
$ {\cal S}_\delta z\in L^\infty(0,T;{\cal C}^\infty_c(\Omega))$. For every $\delta>0$ there exists $N=N(\delta)$ such that  a family of closed
cubes $\{D_{\delta,k}\}_{k=1}^N$ with disjoint interiors and the length of an edge equal to $\delta$ covers $\Omega$, i.e.
$\Omega\subset \bigcup_{k=1}^ND_{\delta,k}$. 
 Hence 
%
\begin{equation}
\int_0^T\int_\Omega M(x,{\cal S}_\delta z(t,x))\,dx=
\sum\limits_{k=1}^{N}
\int_0^T\int_{D_{\delta,k}\cap\Omega}M(x,{\cal S}_\delta z(t,x))\,dx\,dt.
\end{equation}
For each $\delta, k$ by $G_{\delta,k}$ we shall mean a cube with an edge of the length $2\delta$ and centered the same as the corresponding $D_{\delta,k}$. 
Note that if $x\in D_{\delta,k}$, then there exist $2^d$ cubes 
$G_{\delta,k}$ such that $x\in G_{\delta,k}$. 
 Define 
\begin{equation}
m_{k}^\delta(\xi):=\inf_{(t,x)\in ((0,T)\times G_{\delta,k})\cap Q}M(x,\xi)\le 
\inf_{(t,x)\in ((0,T)\times D_{\delta,k})\cap Q}M(x,\xi)
\end{equation}
and 
\begin{equation}
\alpha_{k}(t,x,\delta):=\frac{M(x,{\cal S}_\delta z(t,x))}{m_{k}^\delta({\cal S}_\delta z(t,x))}.
\end{equation}
Then 
\begin{equation}
\int_0^T\int_\Omega M(x,{\cal S}_\delta z(t,x))\,dx\,dt=\sum\limits_{k= 1}^N
\int_0^T\int_{D_{\delta,k}\cap\Omega}\alpha_k(t,x,\delta)m_k^\delta
({\cal S}_\delta z(t,x)) \,dx\,dt.
\end{equation}
We are aiming to estimate the term 
 $\alpha_{k}(t,x,\delta)$ and the main tool here will be the regularity with respect to $x$, which is assumed on M, namely
 condition \eqref{log}. 
 For this purpose let now $(t_k,x_k)$ be the point where the infimum of $M(x,\xi)$ is obtained in the set $(0,T)\times G_{\delta,k}$. Then 
\begin{equation}
\alpha_{k}(t,x,\delta)=\frac{M(x,{\cal S}_\delta z(t,x))}{M(x_k,{\cal S}_\delta z(t,x))}\le |{\cal S}_\delta z(t,x) |^\frac{H}{\ln\frac{1}{|x-x_k|}}.
\end{equation}
Without loss of generality one can assume that 
$\|z\|_{L^\infty(0,T;L^1(\Omega))}\le 1$. By H\"older inequality \eqref{hoelder} we obtain for $\delta<R$
\begin{equation}\label{22}\begin{split}
|{\cal S}_\delta z(t,x)|&\le\left| \frac{1}{\delta^{d}}\left(1-\frac{\delta}{R}\right)^{-1}\sup_{B(0,1)}|S(y)|\int_\Omega\bbbone_{B(0,\delta)}(y)z(t,(1-\frac{\delta}{R})y) \,dy\right|\\
&\le \frac{2}{\delta^{d}}\sup_{B(0,1)}|S(y)|\|z\|_{L^\infty(0,T;L^1(\Omega))}
\le \frac{c}{\delta^{d}}.
\end{split}\end{equation}
Since $x\in D_{\delta,k}$ and $x_k\in G_{\delta,k}$ then   $|x-x_k|\le \delta\sqrt{d}$  and for sufficiently small $\delta$, e.g. $\delta<\frac{1}{2\sqrt{d}}$ 
with  use of  \eqref{22} we obtain 
\begin{equation}\begin{split}
 |{\cal S}_\delta z(t,x) |^\frac{H}{\ln\frac{1}{\delta\sqrt{d}}}&\le(c\delta^{-d})^\frac{H}{\ln\frac{1}
 {\delta\sqrt{d}}}
 \le c^\frac{H}{\ln2} \cdot d^\frac{dH}{\ln 4}\left(e^{\ln \delta\sqrt{d}}\right)^\frac{dH}{\ln\delta\sqrt{d}}
 \le d^\frac{dH}{\ln 4} c^\frac{H}{\ln2}e^{dH}:=C.
\end{split}\end{equation}
Consequently
\begin{equation}\label{osza}
|\alpha_{k}(t,x,\delta)|\le C.
\end{equation}
Define $\tilde M(x,\xi):=\max_{k}m_{k}^\delta(\xi)$ where the maximum is taken with respect to all the sets 
$(0,T)\times G_{\delta,k}$. Obviously  $\tilde M(x,\xi)\le M(x,\xi)$ for all $(t,x)\in Q.$
Using the uniform estimate \eqref{osza} and the Jensen inequality we have 
\begin{equation}\begin{split}
\int_Q &M(x,{\cal S}_\delta z(t,x))dxdy\le C\sum\limits_{k=1}^N
\int_0^T\int_{D_{\delta,k}}m_{k}^\delta({\cal S}_\delta z(t,x)) 
\,dx\,dt\\&
\le C\sum\limits_{k=1}^N\int_{B(0,\delta)}|S_\delta(y)|\,dy
\int_0^T\int_{(1-\frac{\delta}{R})G_{\delta,k}}
m_{k}^\delta( z(t,x))\,dx\,dt\\&
\le 2^{d}C\int_{Q}\tilde M(x,z(t,x)) \,dx\,dt
\le 2^{d}C\int_{Q} M(x,z(t,x)) \,dx\,dt
\end{split}\end{equation}
which completes the proof. 
\end{proof}
\appendix
\section{Auxilary facts}

\begin{lemma}\label{lem-dense}
Let ${\mathbb S}$ be the set of all simple, integrable functions on $Q$ and let \eqref{int} hold. Then ${\mathbb S}$ is dense with respect to the modular topology in $L_M(Q)$. 
\end{lemma}
For the proof in isotropic case see \cite[Theorem 7.6]{Musielak}. The anisotropic case follows exactly the same lines.

Below we formulate some facts concerning convergence in generalized Musielak-Orlicz spaces.
For the proofs of these lemmas and propositions see~\cite{GwSw2008}. 

\begin{lemma}\label{modular-conv}
Let $z^j:Q\to\R^d$ be a measurable sequence. Then
$z^j\modular{M} z$ in $L_M(Q)$ modularly if and only if
$z^j\to z$ in measure and there exist some $\lambda>0$ such that
the sequence $\{M(x,\lambda z^j)\}$ is uniformly integrable in $L^1(Q)$,
i.e.,
$$\lim\limits_{R\to\infty}\left(\sup\limits_{j\in\N}\int_{\{(t,x):|M(x,\lambda z^j)|\ge
R\}}M(x,\lambda z^j)dxdt\right)=0.$$
\end{lemma}
\begin{lemma}\label{uni-int}
Let $M$ be an  ${ N}$--function and for all $j\in\N$ let $\int_Q
M(x,z^j)\,dx\,dt\le c$.
Then the sequence $\{z^j\}$ is
uniformly integrable in $L^1(Q)$.
\end{lemma}
\begin{proposition}\label{product}
Let $M$ be an  ${ N}$--function and $M^\ast$ its complementary
function. Suppose that the sequences $\psi^j:Q\to\R^d$ and
$\phi^j:Q\to\R^d$ are uniformly bounded in $L_M(Q)$ and
$L_{M^\ast}(Q)$ respectively. Moreover $\psi^j\modular{M}\psi$
modularly in $L_M(Q)$ and $\phi^j\modular{M^\ast}\phi$  modularly
in $L_{M^\ast}(Q)$. Then $\psi^j\cdot\phi^j\to\psi\cdot\phi$
strongly in $L^1(Q)$.
\end{proposition}
\begin{proposition}\label{ae}
Let $ K^j$ be a standard mollifier, i.e., $ K\in
C^\infty(\R),$ $ K$ has a compact support and
$\int_\R K(\tau)d\tau=1,  K(t)= K(-t)$. We define
$ K^j(t)=j K(jt).$ Moreover let $\ast$ denote a
convolution in the variable $t$. Then for any  function
$\psi:Q\to\R^d$ such that $\psi\in L^1(Q)$ it holds
$$(\vrho^j\ast\psi)(t,x)\to\psi(t,x)\quad\mbox{in measure}.
$$
\end{proposition}

\begin{proposition}\label{sup}
Let $ K^j$ be defined as in Proposition \ref{ae}. Given an
 ${ N}$--function $M$ and a   function $\psi:Q\to\R^d$ such that
$\psi\in{\mathcal L}_M(Q)$,
the sequence $\{M(\vrho^j\ast\psi)\}$
is uniformly integrable.
\end{proposition}

The next  lemma  is the main tool for showing that  the limits of approximate sequences are in the graph ${\cal A}$ provided that the graph is maximal monotone. This lemma in such a form was formulated in 
\cite{BuGwMaRaSw2012}, see also \cite{Sw2013}. 
\begin{lemma}\label{Minty2}
Let $\mathcal{A}$ be maximal monotone $M$-graph.
 Assume that there are
sequences $\{A^n\}_{n=1}^\infty$ and
$\{\nabla u^n\}_{n=1}^{\infty}$ defined on $Q$ such that 
the following conditions hold:
\begin{align}\label{11}
(\nabla u^n(t,x),A^n(t,x))& \in \mathcal{A}(t,x) &&\textrm{ a.e. in } Q,\\\label{22}
\nabla u^n &\weakstar \nabla u &&\textrm{ weakly}^*  \textrm{in } L_M(Q),\\
A^n &\weakstar A &&\textrm{ weakly}^*  \textrm{in }  L_{M^*}(Q),\label{33}\\\label{44}
\limsup_{n\to \infty} \int_{Q}A^n \cdot \nabla u^n \,d x\,dt &\le \int_{Q} A \cdot \nabla u \, d x\,dt.
\end{align}
Then 
\begin{equation*}
(\nabla u(t,x),A(t,x))\in \mathcal{A}(t,x)\quad \textrm{ a.e. in } Q,\label{Minty2-2}
\end{equation*}
\end{lemma}

\noindent
Finally we summarize some properties of selections. 
\begin{lemma}\label{LS*}
Let $\mathcal{A}(t,x)$ be maximal monotone $M$-graph satisfying {(A1)}--{ (A5)} with measurable selection
$\tilde A:Q\times \mathbb{R}^{d} \to \mathbb{R}^{d}$. Then $\tilde A$
satisfies the following conditions:
\begin{enumerate}
\item [{\it (a1)}]$\mathrm{Dom}\, \tilde A(t,x,\cdot) = \Rd$ a.e. in $Q$;
\item [{\it (a2)}]$\tilde A$ is  monotone, i.e. for every $\xi_1$,
$\xi_2 \in \Rd$ and a.a. $(t,x)\in Q$
\begin{equation} \label{monot}
(\tilde A(t,x,\xi_1) - \tilde A(t,x,\xi_2))\cdot( \xi_1 - \xi_2) \ge 0;
\end{equation}
\item [{\it (a3)}] There are non-negative $k\in
L^1(Q)$,   $c_*>0$ and $N$-function  $M$ such that for all $\nabla u\in\Rd$ the function
 $\tilde A$ satisfies
\begin{equation} \label{growthS*}
\tilde A \cdot \nabla u \geq -k(t,x) +c_*(M(x,\nabla u) +M^*(x,\tilde A))
\end{equation}
\end{enumerate}
Moreover, let
$U$ be a dense set in $\Rd$ and
$(\bB,\tilde A(t,x,\bB)) \in \mA(t,x)$ for a.a. $(t,x)\in Q$ and for all $\bB\in U$. Let also
  $(\nabla u, \bS)\in \Rd \times \Rd$. Then the following conditions are equivalent:
\begin{equation}\begin{split}
\textrm{(i)} \quad &({\bS} - \tilde A(t,x,\bB))\cdot( {\nabla u} - \bB) \geq 0 \quad \textrm{ for all } \quad (\bB,\tilde A(t,x,\bB))\in\mA(t,x)\,,\\
\textrm{(ii)} \quad & (\nabla u, \bS) \in \mA(t,x).
\end{split} \label{maxmon}
\end{equation}
\end{lemma}
For the proof see~\cite{BuGwMaSw2012}.

\bigskip
\noindent
{\bf Acknowledgements }\\
The author was supported by the grant IdP2011/000661.

\noindent

\bibliographystyle{abbrv}
\bibliography{genorlicz}

\end{document}